\documentclass[12pt,reqno]{amsart}

\usepackage{amsfonts,amsthm,amsmath,microtype}
\newcommand{\numberset}{\mathbb} 
\newcommand{\R}{\numberset{R}} 

\newcommand{\varphie}{\varphi_{\varepsilon, i}}
\usepackage[a4paper,top=3cm,bottom=3cm,left=3cm,right=3cm,bindingoffset=0mm]{geometry}
\newcommand{\Uk}{\bigcup_{k=1}^n B(a_k,r_0)}
\usepackage{verbatim}
\usepackage{setspace}

\makeatletter

\numberwithin{equation}{section}
\newtheorem{thm}{\indent\bf {Theorem}}[section]
\newtheorem{lemma} [thm] {\indent\bf Lemma}

\newtheorem{prop}[thm]{\indent\bf Proposition}

\theoremstyle{definition}

\def\author@andify{%
   \nxandlist {\unskip ,\penalty-1 \space\ignorespaces}%
     {\unskip {} }%
     {\unskip ,\penalty-2 \space }%
}
\title{Weighted multipolar Hardy inequalities and evolution problems with Kolmogorov operators 
perturbed by singular potentials}

\author[A. Canale]{Anna Canale}
\address{Dipartimento di Ingegneria dell'Informazione ed Elettrica e Matematica Applicata (Diem), 
Universit\'a degli Studi di Salerno, Via Giovanni Paolo II, 132, 84084 Fisciano
(Sa), Italy.}
\email{acanale@unisa.it}

\author[F.Pappalardo]{Francesco Pappalardo}
\address{Dipartimento di Matematica e Applicazioni \textquotedblleft Renato  Caccioppoli\textquotedblright,
Universit\'a degli Studi di Napoli Federico II, Complesso Universitario Monte S. Angelo, Via Cintia, 80126 Napoli, Italy.}
\email{francesco.pappalardo@unina.it}

\author[C. Tarantino]{Ciro Tarantino}
\address{Dipartimento di Scienze Economiche e Statistiche, 
Universit\'a degli Studi di Napoli Federico II, Complesso Universitario Monte S. Angelo,
 Via Cintia, 80126 Napoli, Italy.}
\email{ctarant@unina.it}

\thanks{The first two authors are members of the Gruppo Nazionale per l'Analisi Matematica, la Probabilit\'a e le loro Applicazioni 
(GNAMPA) of the Istituto Nazionale di Alta Matematica (INdAM)}


\begin{document}
\maketitle
\begin{abstract}

The main results in the paper are the weighted multipolar Hardy inequalities
\begin{equation*}
c\int_{\R^N}\sum_{i=1}^n\frac{u^2}{|x-a_i|^2}\,d\mu
\leq\int_{\R^N}|\nabla u |^2d\mu+
K\int_{\R^N} u^2d\mu,
\end{equation*}
in $\R^N$ for any $u$ in a suitable weighted Sobolev space, with $0<c\le 
c_{o,\mu}$, $a_1,\dots,a_n\in \R^N$, $K$ constant.
The weight functions $\mu$ are of a quite general type. 

The paper fits in the framework of the study of Kolmogorov operators
\begin{equation*}
Lu=\Delta u+\frac{\nabla \mu}{\mu}\cdot\nabla u,
\end{equation*}
perturbed by multipolar inverse square potentials, and of the related evolution problems.

The necessary and sufficient conditions for the existence 
of positive exponentially bounded in time solutions to the associated initial 
value problem are based on weighted Hardy inequalities.
The optimality of the constant constant $c_{o,\mu}$ allow us to state the nonexistence of positive solutions. 

We follow the Cabr\'e-Martel's approach. To this aim we state some properties
 of the operator $L$, of its corresponding $C_0$-semigroup and density results.

\end{abstract}

\maketitle

\bigskip


\bigskip

\section{Introduction}
\label{Introduzione}

The paper concerns the weighted multipolar Hardy inequalities in $\R^N$
for a class of weight functions $\mu$. 
The main motivation for our interest in 
Hardy inequalities is the key role that these play in the study of Kolmogorov operators 
\begin{equation}\label{L Kolmogorov}
Lu=\Delta u+\frac{\nabla \mu}{\mu}\cdot\nabla u,
\end{equation}
defined on smooth functions,
perturbed by singular potentials and of the related evolution problems
$$
(P)\quad \left\{\begin{array}{ll}
\partial_tu(x,t)=Lu(x,t)+V(x)u(x,t),\quad \,x\in {\mathbb R}^N, t>0,\\
u(\cdot ,0)=u_0\geq 0\in L_\mu^2
\end{array}
\right. $$
where $L_\mu^2:=L(\R^N, d\mu)$,  with $d\mu(x)=\mu(x)dx$, 
$0\le V\in L_{loc}^1(\R^N)$.

The potentials we consider are inverse square potentials of multipolar type 
\begin{equation}\label{V}
V(x)=\sum_{i=1}^n \frac{c}{|x-a_i|^2},\quad c>0, \quad a_1\dots,a_n\in \R^N.
\end{equation}
In literature there exist reference papers 
in the case of Schr\"odinger operators with singular potentials of the type 
$V(x)\sim\frac{c}{|x|^2}$, $c>0$. These potentials are interesting for the criticality:
the strong maximum principle and Gaussian bounds fail (see \cite{Aronson}). 

The operator $\Delta+\frac{c}{|x|^{2}}$ has the same homogeneity as the Laplacian.
In 1984 by P. Baras and J. A. Goldstein in \cite{BarasGoldstein}
showed that the evolution problem $(P)$ with $L=\Delta$
admits a unique positive solution
if $c\leq c_o=\left( \frac{N-2}{2} \right)^{2}$ and no positive solutions exist if $c>c_o$.
When it exists, the solution is
exponentially bounded, on the contrary, if $c>c_o$, there is the so-called instantaneous blowup phenomenon.

The drift term in (\ref{L Kolmogorov}) forces to use a different technique 
in order to extend these results to Kolmogorov operators.

A result analogous to that stated in  \cite{BarasGoldstein} has been obtained in 1999 by 
X. Cabr\'e and Y. Martel in \cite{CabreMartel}  for more general potentials 
$0\le V\in L_{loc}^1(\R^N)$ with a different approach.

To state existence and nonexistence results we follow the Cabr\'e-Martel's approach
using the relation between the weak solution of $(P)$
and the {\it bottom of the spectrum} of the operator $-(L+V)$
\begin{equation*}
\lambda_1(L+V):=\inf_{\varphi \in H^1_\mu\setminus \{0\}}
\left(\frac{\int_{{\mathbb R}^N}|\nabla \varphi |^2\,d\mu
-\int_{{\mathbb R}^N}V\varphi^2\,d\mu}{\int_{{\mathbb R}^N}\varphi^2\,d\mu}
\right),
\end{equation*}
with $H^1_\mu$ suitable weighted Sobolev space.

When $\mu=1$ Cabr\'e and Martel 
showed that the boundedness  of 
 $\lambda_1(\Delta+V)$ is a necessary 
and sufficient condition for the existence of positive exponentially bounded in time
solutions to the associated initial value problem. 
Later in \cite{GGR, CGRT, CPT1} similar results have been extended to Kolmogorov operators
perturbed by inverse square potentials with a single pole.
The proof uses some properties of the operator $L$ and of its corresponding semigroup
in $L_\mu^2(\R^N)$. 

In the multipolar case with $L=\Delta$ the behaviour of the operator with a multipolar inverse square potential
has been investigated in literature.
In particular if $\mathcal{L}$ is the Schr\"odinger operator
$$
\mathcal{L}=-\Delta-\sum_{i=1}^n\frac{c_i}{|x-a_i|^2},
$$
$n\ge2$, $c_i\in \R$, for any $i\in \{1,\dots, n\}$, V. Felli, E. M. Marchini and S. Terracini in
\cite{FelliMarchiniTerracini} proved that the associated quadratic form
$$
Q(\varphi):=\int_{\R^N}|\nabla \varphi |^2\,dx
-\sum_{i=1}^n c_i\int_{{\mathbb R}^N}\frac{\varphi^2}{|x-a_i|^2}\,dx
$$
is positive if $\sum_{i=1}^nc_i^+<\frac{(N-2)^2}{4}$, $c_i^+=\max\{c_i,0\}$, 
conversely if
$\sum_{i=1}^nc_i^+>\frac{(N-2)^2}{4}$ there exists a configuration of poles such that $Q$ is not positive.
Later Bosi, Dolbeaut and Esteban in \cite{BDE} proved that for any $c\in\left(0,\frac{(N-2)^2}{4}\right]$ there exists 
a positive constant $K$ such that a multipolar Hardy inequality holds.
Cazacu and Zuazua in \cite{CazacuZuazua}, improving a result stated in 
\cite{BDE}, obtained the inequality when 
$V= c\sum_{1\le i<j\le n}\frac{|a_i-a_j|^2}{ |x-a_i|^2|x-a_j|^2}$ (see also Cazacu \cite{Cazacu} 
for estimates for the Hardy constant in bounded domains).

For Ornstein-Uhlenbeck type operators 
$$
Lu=\Delta u - \sum_{i=1}^{n}A(x-a_i)\cdot \nabla u,
$$
with $A$ a positive definite real Hermitian $N\times N$ matrix, $a_i\in \R^N$, $i\in \{1,\dots , n\}$,  perturbed by multipolar
inverse square potentials (\ref{V}),
weighted multipolar Hardy inequalities
and related existence and nonexistence results were stated in \cite{CP}. 
In such a case, the invariant measure for these operators is the Gaussian measure
$d\mu =\mu_A (x) dx =\kappa e^{-\frac{1}{2}\sum_{i=1}^{n}\left\langle A(x-a_i), x-a_i\right\rangle }dx$.  

As far as we know there are no other results in the literature about the weighted multipolar Hardy inequalities.

In the paper, in Sections 2 and 3, we state multipolar weighted inequalities
\begin{equation}\label{whi}
\int_{\R^N}V\,\varphi^2\,d\mu
\leq\int_{\R^N}|\nabla \varphi |^2d\mu+
K\int_{\R^N} \varphi^2d\mu,\quad \varphi\in H^1_\mu,\qquad K>0,
\end{equation}
with $V$ as in (\ref{V}), with $0<c\le c_{o,\mu}$, 
and state the optimality of the constant on the left-hand side.

We use two different approaches to get the estimates. 
The first is based on the well known {\sl vector field method} and 
the second extends the {\sl IMS method} used in \cite {BDE} to the weighted case.

There is a close relation between
the estimate of the bottom of the spectrum $\lambda_1(L+V)$ and the weighted Hardy inequalities.
In particular the existence of positive solutions to $(P)$
is related to the Hardy inequality (\ref{whi})  and the nonexistence 
is due to the optimality of the constant $c_{o, \mu}$. 

The main difficulties to get the inequality in the multipolar case are due to the mutual interaction among the poles.
In \cite{CP} we used a technique which allowed us to overcome such difficulties in the case of
the Gaussian measure, but it does not work in the setting of more general measures.

It is not immediate to generalize the vector field method to the multipolar case.
In order to do this, we need to isolate the poles. We are able to attain the result with assumptions on the weights
which generalize in a natural way those in the unipolar case (cf. \cite{CPT1}). 
The limit of the method is that we do not achieve the best constant $c_{o,\mu}$ 
on the left hand side in the estimate.

The IMS method allows us to get the best constant. Up to now this is the unique technique which allows 
to achieve the optimal constant in the case of Lebesgue measure (cf. \cite{BDE}).
We adapt the method to the weighted case. 

The technique makes use of a weighted Hardy inequality 
with a single pole. In the weighted case
the assumptions on $\mu$ must allow us to use an unipolar estimate with the same measure. 
This is a disadvantage compared to the first method and it forces us to use
assumptions on $\mu$ which are a bit less general. 
Good weight functions $\mu$ are the ones that behave in a unipolar way near to the single pole.
We use as a suitable inequality the unipolar inequality stated in \cite{CPT1}.

A class of functions satisfying our hypotheses is shown in Section 4. 

In Section 5 we get the optimality of the constant in the estimate. 
A crucial point is to find a suitable function $\varphi$ for which the inequality
(\ref{whi}) doesn't hold if $c>c_{o,\mu}$. We present a function which involves only one pole
reasonig as in \cite{CPT1}. Furthermore we adapt the way to estimate the bottom of the spectrum
in  \cite{CGRT} to the multipolar case.

We state existence and nonexistence result in Section 6 following the Cabr\'e-Martel's approach 
and, then, using multipolar weighted inequalities. 
So we need that the unperturbed operator $L$ generates a $C_o$-semigroup. 
In the case of measures of a more general type than the Gaussian one, 
measures which could have degeneracy in one or more points, we need to require
suitable assumptions to guarantee the generation of the semigroup. 

The proof of Theorem \ref{th:cabr-mart-H1} relies on certain properties of the operator $L$ 
and of its corresponding semigroups.
We ensure that these properties hold reasoning as in \cite{CGRT}. To this aim we state
some density results.

\bigskip

\section{Weighted multipolar Hardy inequalities via the vector field method}\label{vector-field}

Let $\mu\ge 0$ be a weight function on $\R^N$.
The vector field method suggests us to consider the vectorial function
$$
F(x)=\sum_{i=1}^n \beta\, \frac{x-a_i}{|x-a_i|^2} \mu, \qquad \beta>0.
$$
Let us assume the following hypotheses  
\begin{itemize}
\item[$H_1)$] 
\begin{itemize}
\item[$i)$] $\quad \nabla \mu \in L_{loc}^1(\R^N)$;
\item[$ii)$] $\quad \sqrt{\mu}\in H^1_{loc}(\R^N)$;
\item[$iii)$]  $\quad \mu^{-1}\in L_{loc}^1(\R^N)$;
\end{itemize}
\item [$H_2)$] there exists constants $k_1,k_2\in \R$, $k_2>2-N$, such that
$$
\beta\sum_{i=1}^n\frac{(x-a_i)}{|x-a_i|^2}\cdot\nabla\mu\ge
\left( -k_1 + \sum_{i=1}^n\frac{k_2\beta}{|x-a_i|^2}\right) \mu;
$$
\end{itemize}
Let us observe that under the assumptions $ii)$ and $iii)$ in the hypothesis $H_1)$ the 
space $C_c^{\infty}(\R^N)$ is dense in $H_{\mu}^1$
and $H_{\mu}^1$ is the completion of $C_c^{\infty}(\R^N)$ with respect to the Sobolev norm
$$
\|\cdot\|_{H^1_\mu}^2 := \|\cdot\|_{L^2_\mu}^2 + \|\nabla \cdot\|_{L^2_\mu}^2
$$
(see e.g. \cite{T}).
\medskip

\begin{thm}\label{wH vector field}
Let $\displaystyle r_0=\min_{i\neq j} |a_i-a_j|/2$, $N\ge 3$, $n\ge 1$.
Under hypotheses  $H_1)$ and $H_2)$ we get
\begin{equation}\label{ineq vector field}
\begin{split}
\frac{c_o(N+k_2)}{n}\int_{{\mathbb R}^N}\sum_{i=1}^n \frac{\varphi^2 }{|x-a_i|^2}\, d\mu
\,+\, &
\frac{\beta^2}{2}\int_{\R^N}\sum_{\substack{i,j=1\\ i\ne j}}^{n}
\frac{|a_i-a_j|^2}{|x-a_i|^2|x-a_j|^2}\varphi^2\,d\mu  
\\&\le
\int_{{\mathbb R}^N} |\nabla\varphi|^2 \, d\mu +k_1\int_{{\mathbb R}^N}\varphi^2 d\mu,
\end{split}
\end{equation}
for any $ \varphi \in H_\mu^1$, where  $c_o(N+k_2):=\left(\frac{N+k_2-2}{2}\right)^2$.
\end{thm}

\begin{proof}

By density, it is enough to prove (\ref{ineq vector field}) for every 
$\varphi \in C_{c}^{\infty}(\R^N)$. 

It is immediate to verify that
\begin{equation}\label{def div F}
\int_{\R^N}\varphi^2 {\rm div}F \,dx=
\beta\int_{\R^N}\sum_{i=1}^n\left[\frac{N-2}{|x-a_i|^2} \mu +
\frac{(x-a_i)}{|x-a_i|^2}\cdot\nabla\mu\right]\varphi^2 dx.
\end{equation} 
On the other hand, integrating by parts and using H{\"o}lder and Young inequalities, 
we get
\begin{equation}\label{div F}
\begin{split}
\int_{\R^N}&\varphi^2 {\rm div}F \,dx=
-2\int_{\R^N}^{}\varphi F\cdot\nabla\varphi \, dx
\\&
\le 2\left[\int_{\R^N}|\nabla \varphi|^2 \, d\mu\right]^{\frac{1}{2}}
\left[\int_{\R^N}\left( \sum_{i=1}^{n} \frac{ \beta\,(x-a_i)}{|x-a_i|^2} \right)^2 
\,\varphi^2\, d\mu\right]^{\frac{1}{2}}
\\&
\le \int_{\R^N}|\nabla \varphi|^2 \, d\mu+
\int_{\R^N}\left( \sum_{i=1}^{n} \frac{\beta\,(x-a_i)}{|x-a_i|^2} \right)^2 \,\varphi^2\, d\mu.
\end{split}
\end{equation}
From (\ref{def div F}) and (\ref{div F}) we deduce
\begin{equation}\label{first part}
\begin{split}
\int_{\R^N}\sum_{i=1}^n \frac{\beta(N-2)}{|x-a_i|^2}\varphi^2d\mu
\le & \int_{\R^N}|\nabla \varphi|^2 \, d\mu
\\&+
\int_{\R^N} \sum_{i=1}^{n} \frac{\beta^2}{|x-a_i|^2}  \,\varphi^2\, d\mu
\\&
+\int_{\R^N} \sum_{\substack{i,j=1\\i\ne j}}^n\frac{ \beta^2\,(x-a_i)\cdot (x-a_j)}{|x-a_i|^2|x-a_j|^2} 
\,\varphi^2\, d\mu
\\&
-\beta
\int_{\R^N}\sum_{i=1}^n\frac{(x-a_i)}{|x-a_i|^2}\cdot\nabla\mu\,\varphi^2 dx.
\end{split}
\end{equation}
Now we observe that
\begin{equation}\label{by double product}
\begin{split}
\sum_{\substack{i,j=1\\ i\ne j}}^{n}\frac{(x-a_i)\cdot (x-a_j)}{|x-a_i|^2|x-a_j|^2}  &=
\sum_{\substack{i,j=1\\ i\ne j}}^{n}\frac{|x|^2-xa_i-xa_j+a_ia_j}
{|x-a_i|^2|x-a_j|^2}
\\&=
\sum_{\substack{i,j=1\\ i\ne j}}^{n}\frac{\frac{|x-a_i|^2}{2}+\frac{|x-a_j|^2}{2}-\frac{|a_i-a_j|^2}{2}}
{|x-a_i|^2|x-a_j|^2}
\\&=
\sum_{\substack{i,j=1\\ i\ne j}}^{n}\frac{1}{2}\left( \frac{1}{|x-a_i|^2}+\frac{1}{|x-a_j|^2}
-\frac{|a_i-a_j|^2}{|x-a_i|^2|x-a_j|^2}\right) 
\\&=
(n-1)\sum_{i=1}^{n}\frac{1}{|x-a_i|^2}
-\frac{1}{2}\sum_{\substack{i,j=1\\ i\ne j}}^{n}\frac{|a_i-a_j|^2}{|x-a_i|^2|x-a_j|^2}.
\end{split}
\end{equation}
Then, taking into account the hypothesis $H_2)$ and using (\ref{by double product}), 
from the estimate (\ref{first part}) it follows that
\begin{equation}\label{part on left side}
\begin{split}
\left[ (N+k_2-2)\beta-n\beta^2 \right]&\sum_{i=1}^{n} 
\int_{\R^N}\frac{\varphi^2}{|x-a_i|^2}\,d\mu
\\+&
\frac{\beta^2}{2}\int_{\R^N}\sum_{\substack{i,j=1\\ i\ne j}}^{n}
\frac{|a_i-a_j|^2}{|x-a_i|^2|x-a_j|^2}\varphi^2\,d\mu
\\&\le 
\int_{\R^N}|\nabla \varphi|^2\,d\mu +k_1\int_{\R^N}\varphi^2 d\mu.
\end{split}
\end{equation}
The Theorem is proved observing that
$$
\max_\beta[(N+k_2-2)\beta-n\beta^2]=\frac{(N+k_2-2)^2}{4 n}.
$$
\end{proof}

\medskip

Now our aim is to estimate the second term on the left hand side in
(\ref{part on left side}) to get a more general Hardy inequality.
From a mathematical point of view the principal problem is due to the square of the sum 
on the right-hand side in (\ref{div F}). To overcome the difficulties we 
 are able to isolate singularities but we can not achieve the constant 
$c_o(N+k_2)$.

We state the following result.

\begin{thm}\label{wH vector field 2}
Let $\displaystyle r_0=\min_{i\neq j} |a_i-a_j|/2$, $N\ge 3$, $n\ge 1$.
Then 
if conditions $H_1)$ and $H_2)$ hold, we get
\begin{equation}\label{ineq vector field 2}
c\int_{{\mathbb R}^N}\sum_{i=1}^n \frac{\varphi^2 }{|x-a_i|^2}\, d\mu\le  
\int_{{\mathbb R}^N} |\nabla\varphi|^2 \, d\mu 
+K \int_{\R^N}\varphi^2 \, d\mu
\end{equation}
for any $ \varphi \in H^1_\mu$, where $c\in\left]0,c_o(N+k_2)\right[$,
$c_o(N+k_2)=\left(\frac{N+k_2-2}{2}\right)^2$, and
$K=K(n, c, r_0)$.
\end{thm}

\begin{proof}

Arguing as in the proof of Theorem \ref{wH vector field} (cf. (\ref{first part})) we get 
\begin{equation}\label{I2345}
\begin{split}
\int_{\R^N}\sum_{i=1}^n &\frac{\beta(N-2)}{|x-a_i|^2}\varphi^2d\mu
\le \int_{\R^N}|\nabla \varphi|^2 \, d\mu
\\&+
\int_{\R^N} \sum_{i=1}^{n} \frac{\beta^2}{|x-a_i|^2}  \,\varphi^2\, d\mu
\\&+
\int_{\R^N\setminus \Uk} \sum_{\substack{i,j\\i\ne j}}
\frac{ \beta^2\,(x-a_i)\cdot (x-a_j)}{|x-a_i|^2|x-a_j|^2} \,\varphi^2\, d\mu
\\&
+\int_{\Uk}\sum_{\substack{i,j\\i\ne j}}
\frac{ \beta^2\,(x-a_i)\cdot (x-a_j)}{|x-a_i|^2|x-a_j|^2}  \,\varphi^2\, d\mu
\\&-
\beta
\int_{\R^N}\sum_{i=1}^n\frac{(x-a_i)}{|x-a_i|^2}\cdot\nabla\mu\,\varphi^2 dx
\\&
=:I_1+I_2+I_3+I_4+I_5,
\end{split}
\end{equation}
where $B(a_k,r_0)$, $k=1,\dots, n$, denotes the open ball of $\R^N$ of radius $r_0$ centered at $a_k$.

Let us estimate $I_3$ and $I_4$. The first integral can be estimate as follows 

\begin{equation}\label{I3} 
I_3\le
\frac{ \beta^2}{r_0^2} n(n-1)\int_{\R^N\setminus \Uk}\varphi^2 \, d\mu.
\end{equation}

For the second integral we isolate the singularities and then, using again Young inequality, we get
\begin{equation}\label{I4}
\begin{split}
I_4 & \le 
\sum_{k=1}^{n} \Biggl(
\int_{B(a_k,r_0)}\sum_{\substack{j=1\\j\ne k}}^n\frac{\beta^2}{|x-a_k||x-a_j|} \,\varphi^2\, d\mu+
\\&\quad +
\int_{B(a_k,r_0)}\sum_{\substack{i,j=1\\j\ne i\ne k}}^n\frac{\beta^2}{|x-a_i||x-a_j|} 
\,\varphi^2\, d\mu\Biggr)
\\&
\le\sum_{k=1}^{n}\Biggl\{
\frac{\epsilon}{2}
\int_{B(a_k,r_0)}\frac{\beta^2}{|x-a_k|^2}\,\varphi^2 \, d\mu+
\\&\quad +
\frac{1}{2\epsilon}
\int_{B(a_k,r_0)}\sum_{\substack{j=1\\j\ne k}}^n
\frac{\beta^2}{|x-a_j|^2}\,\varphi^2\, d\mu
\\&\quad +
\frac{\beta^2}{r_0^2}(n-1)^2\int_{B(a_k,r_0)}\varphi^2 \, d\mu\Biggr\}
\\&\le
\sum_{k=1}^{n}\Biggl\{
\frac{\epsilon}{2}
\int_{B(a_k,r_0)}\frac{\beta^2}{|x-a_k|^2}\,\varphi^2\, d\mu+
\\&
\quad +\Biggl[\frac{\beta^2(n-1)}{2\epsilon\, r_0^2}+\frac{\beta^2(n-1)^2}{r_0^2}\Biggr]
\int_{B(a_k,r_0)}\varphi^2 \, d\mu\Biggr\}
\\&\le
\sum_{k=1}^{n}\Biggl\{
\frac{\epsilon}{2}
\int_{B(a_k,r_0)}\sum_{i=1}^n\frac{\beta^2}{|x-a_i|^2}\,\varphi^2\, d\mu+
\\&
\quad +\frac{\beta^2(n-1)}{ r_0^2}\biggl[\frac{1}{2\epsilon}+
(n-1)\Biggr]
\int_{B(a_k,r_0)}\varphi^2 \, d\mu\biggr\}.
\end{split}
\end{equation}
The integral $I_5$ can be estimate applying $H_2)$.

Taking into account (\ref{I2345}) and using (\ref{I3}), (\ref{I4}) we deduce that
\begin{equation}\label{c alpha}
\begin{split}
\int_{\R^N}\sum_{i=1}^{n}&
\frac{\beta(N+k_2-2)  -\beta^2(1+\frac{\epsilon}{2})}{|x-a_i|^2}\varphi^2 \, d\mu
\\&\le
\int_{\R^N}^{}|\nabla \varphi|^2 d\mu 
+K \int_{\R^N} \varphi^2\, d\mu,
\end{split}
\end{equation}
where 
$$K=\frac{\beta^2}{r_0^2}(n-1)\left(n-1+\frac{1}{2\epsilon}\right)+k_1.
$$
The maximum of the function $\beta\mapsto (N+k_2-2)\beta-\beta^2(1+\frac{\epsilon}{2})$ 
is $\frac{c_o(N+k_2)}{1+\frac{\epsilon}{2}}$
attained in $\beta_{max}=\frac{\sqrt{c_o(N+k_2)}}{1+\frac{\epsilon}{2}}$.
So, if we set 
\begin{equation}\label {c alpha}
c=(N+k_2-2)\beta-\beta^2\left(1+\frac{\epsilon}{2}\right)
\end{equation}
we deduce from (\ref{c alpha}) that for $c\in\left(0,\frac{c_o(N+k_2)}{1+\frac{\epsilon}{2}}\right]$, 
for any $\epsilon>0$, it holds
\begin{equation*}
c\int_{\R^N}\sum_{i=1}^n \frac{\varphi^2}{|x-a_i|^2}\,d\mu
\le \int_{\R^N} |\nabla\varphi|^2\,d\mu +
K\int_{\R^N}\varphi^2 d\mu.
\end{equation*}
The relation (\ref{c alpha}) between $\beta$ and $c$ allow us to write $\beta$ in the following form
$$
\beta_{\epsilon}^{\pm}=\frac{\sqrt{c_o(N+k_2)}\pm
\sqrt{c_o(N+k_2)-c[1+\frac{\epsilon}{2}]}}{1+\frac{\epsilon}{2}}.
$$
\end{proof}

\bigskip

\section{Weighted multipolar Hardy inequalities via the IMS method}\label{ims}

In this Section we state the weighted multipolar Hardy inequality using the so-called 
IMS truncation method (for Ismagilov, Morgan, Morgan-Simon, Sigal, see \cite{Morgan, SimonIMS}), 
which consists 
in localizing the wave functions around the singularities by using a partition of unity. This method, 
unlike the vector field one, allows us to achieve the constant on the left-hand side in the 
inequality.

We argue as in \cite{BDE} adapting the proof to the weighted case. 

\medskip

The hypotheses on the weight functions $\mu$ are $H_1)$ 
in Section \ref{vector-field} and the following

\begin{itemize}

\item[$H_2')$] there exist constants $k_1, k_2\in \R$, $k_2>2-N$, such that if
$$
f_{\varepsilon, i}=(\varepsilon +|x-a_i|^{2})^{\frac{\alpha}{2}}, 
 \quad \alpha < 0, \quad \varepsilon >0,
$$
it holds
\begin{equation*} 
\frac{\nabla f_{\varepsilon, i}}{f_{\varepsilon, i}} \cdot \nabla\mu =
\frac{\alpha(x-a_i)}{\varepsilon +|x-a_i|^{2}}\cdot \nabla\mu 
\le \left( k_1 
+ \frac{ k_2\alpha}{\varepsilon +|x-a_i|^2}\right) \mu
\quad {\rm in} \,\,  B(a_i, r_0)
\end{equation*}
for any $i=1,\dots,n $, and for any $\varepsilon >0$.
\end{itemize}
Under these conditions the weighted unipolar Hardy inequality stated in \cite{CPT1} holds with 
respect to any single pole 
$a_i$, $i=1,\dots,n$,
\begin{equation}\label{wHi-c_i}
c \int_{\R^N} \frac{\varphi^2}{|x-a_i|^2}\,d\mu \le \int_{\R^N}|\nabla \varphi|^2\,d\mu 
+k_1\int_{\R^N}\varphi^2\,d\mu,
\end{equation}
for any function $\varphi \in H_{\mu}^1$, where $c\in(0,c_o(N+k_2)]$ with
$c_o(N+k_2)=\left( \frac{N+k_2-2}{2}\right)^2 $.
Such an estimate plays a fundamental role in the proof of the multipolar Hardy inequality. 

The statement of our inequality is the
following.

\medskip

\begin{thm}\label{Hardy via IMS}
Assume hypotheses $H_1)$ and $H_2')$.   
Let $N\ge 3$, $n\ge 2$ and $\displaystyle r_0=\min_{i\neq j} |a_i-a_j|/2$, $i,j=1,\dots,n$. 
Then there exists
a constant $k_0\in [0, \pi^2)$ such that
\begin{equation}\label{MwhiGeneral}
\begin{split}
c\int_{{\R}^N}\sum_{i=1}^n\frac{\varphi^2 }{|x-a_i|^2}\, d\mu \le &  
\int_{{\R}^N} |\nabla\varphi|^2 \, d\mu 
\\&+
 \left[\frac{k_0+(n+1)c}{ r_0^2}+k_1\right] \int_{\R^N}\varphi^2 \, d\mu ,
\end{split}
\end{equation}
for all $\varphi \in H^1_\mu$, where $c\in (0, c_o(N+k_2)]$ with
$c_o(N+k_2)=\left( \frac{N+k_2-2}{2}\right)^2 $.
\end{thm}

In order to prove the Theorem via the IMS method, we need to recall the notion of partition of unity 
and some related lemmas. 

\medskip

We say that a finite family
$\left\lbrace J_i \right\rbrace_{i=1}^{n+1} $ 
of real valued functions
$J_i\in W^{1,\infty}(\R^N)$ 
is a \textit{partition of unity} in $\R^N$ if $\sum_{i=1}^{n+1}J_i^2=1$.
Furthermore we require that
\begin{equation}\label{omega}
\Omega_i\cap\Omega_j=\emptyset \quad\text{for any } i,j= 1,\dots,n ,\, i\neq j,
\end{equation}
where $\overline{\Omega}_i={\rm supp}(J_i)$, $ i=  1,\dots,n $.

Any family of this type has the following properties:
\begin{enumerate}
\item[a)] $\sum_{i=1}^{n+1}J_i\partial_\alpha J_i=0$ for any $\alpha= 1,\dots,N $;
\item[b)] $J_{n+1}=\sqrt{1-\sum_{i=1}^{n}J_i^2}$;
\item[c)] $\sum_{i=1}^{n+1}|\nabla J_i|^2\in L^{\infty}(\R^N)$;
\item[d)] $\sum_{i=1}^{n+1}|\nabla J_i|^2=\sum_{i=1}^{n}\frac{|\nabla J_i|^2}{1-J_i^2}$.
\end{enumerate}
Note that to avoid a singularity for the gradient of $J_{n+1}$ at the points where 
$1-J_i^2=0$, from d) we shall assume
the additional constraint $|\nabla J_i|^2=F(x)(1-J_i^2)$, for $i= 1,\dots,n $ 
and for some $F\in L^{\infty}(\R^N)$.

\medskip

By proceeding as in \cite[Lemma 2]{BDE}, 
we are able to state the following result.

\medskip

\begin{lemma}\label{lemma2BDE}
Let $\left\lbrace J_i \right\rbrace_{i=1}^{n+1} $ be a partition of unity satisfying (\ref{omega}). 
For any $\varphi\in H^1_\mu$ and any $V\in L^1_{loc}(\R^N)$ we get
\begin{equation*}
\begin{split}
\int_{\R^N}\left( |\nabla \varphi|^2 -V\varphi^2 \right)d\mu = &
 \sum_{i=1}^{n+1}\int_{\R^N}( |\nabla (J_i \varphi)|^2 - V(J_i \varphi)^2 ) d\mu\\
 & -\int_{\R^N}\sum_{i=1}^{n+1} |\nabla J_i|^2\varphi^2\,d\mu.
\end{split}
\end{equation*}
\end{lemma}
\begin{proof}
We can immediately observe that
\begin{equation}\label{primo pezzo}
\begin{split}
\int_{\R^N} V \left( \sum_{i=1}^{n+1}(J_i\varphi)^2\right)\,d\mu
=&\int_{\R^N} V \left( \sum_{i=1}^{n+1}J_i^2\right)
\varphi^2\,d\mu \\
=&\int_{\R^N} V\varphi^2\,d\mu.
\end{split}
\end{equation}
On the other hand,
\begin{equation}\label{calcoloNabla}
\begin{split}
\sum_{i=1}^{n+1}|\nabla \left(J_i\varphi \right)|^2 &=\sum_{i=1}^{n+1}|(\nabla J_i)\varphi 
+ (\nabla \varphi)J_i|^2\\
&=\sum_{i=1}^{n+1}|\nabla J_i|^2\varphi^2
+\sum_{i=1}^{n+1}|\nabla \varphi|^2J_i^2\\
&\quad +2\sum_{i=1}^{n+1}(J_i\nabla J_i) 
(\varphi \nabla \varphi) \\
&=\sum_{i=1}^{n+1}|\nabla J_i|^2\varphi^2+|\nabla \varphi|^2
+\left( \sum_{i=1}^{n+1}J_i\nabla J_i \right) \nabla \varphi^2.
\end{split}
\end{equation}
By property a) it follows that $\left( \sum_{i=1}^{n+1}J_i\nabla J_i \right) \nabla \varphi^2=0$, 
then by integrating 
(\ref{calcoloNabla}) on $\R^N$ we obtain
\begin{equation}\label{secondo pezzo}
\int_{\R^N}|\nabla \varphi|^2\,d\mu
=\int_{\R^N}\sum_{i=1}^{n+1}|\nabla \left(J_i\varphi \right)|^2\,d\mu-\int_{\R^N} 
\sum_{i=1}^{n+1} |\nabla J_i|^2 \varphi^2\,d\mu.
\end{equation}
From (\ref{primo pezzo}) and (\ref{secondo pezzo}) we get the result.
\end{proof}

\medskip

In the following we set 
$$V_n(x)=\sum_{i=1}^{n}\frac{1}{|x-a_i|^2}.$$ 
We recall a preliminary Lemma, stated in \cite{BDE}, about the case $n=2$, 
with $a_1=a$, $a_2=-a$ and $0<r_0\le |a|$.
\begin{lemma}\label{lemma3BDE}
There is a partition of the unity $\left\lbrace J_i \right\rbrace_{i=1}^3$ 
satisfying (\ref{omega}) with $J_1\equiv 1$ on 
$B(a,\frac{r_0}{2})$, $J_1\equiv 0$ on $B(a,r_0)^c$, $J_2(x)=J_1(-x)$ for any 
$x\in \R^N$, $0<r_0\le |a|$, such that, for any $c>0$,
 there exists a constant $k_0\in \left[ 0, \pi^2 \right) $ for which, almost 
 everywhere for all $x\in \Omega:=
{\rm supp}(J_1)\cup{\rm supp}(J_2)$, we have 
\begin{equation}
\sum_{i=1}^{3}|\nabla J_i|^2 + c\,J_3^2\,V_2(x)
=\sum_{i=1,2} \frac{|\nabla J_i|^2}{1-J_i^2}+ 
c\,J_3^2\,V_2(x)\le \frac{k_0+2c}{r_0^2}.
\end{equation}
\end{lemma}

\medskip

As observed in \cite{BDE}, a partition of unity satisfying the hypotheses of Lemma \ref{lemma3BDE} 
is given by setting 
\begin{equation}\label{PartitionJ}
J(t):= \left\lbrace 
\begin{array}{ll}
1 &\,\text{ if } t\le 1/2\\
\sin(\pi t) &\,\text{ if } 1/2\le t\le 1\\
0 &\,\text{ if } t\ge 1
\end{array}
\right. 
\end{equation}
and defining $J_1(x):=J(|x-a|/r_0)$, $J_2(x):=J(|x+a|/r_0)$, and $J_3(x):=\sqrt{1-J_1^2-J_2^2}$.

\medskip

Now we are able to proceed with the proof of inequality (\ref{MwhiGeneral}).

\medskip

\begin{proof}[Proof of Theorem \ref{Hardy via IMS}]
Let us define the following quadratic form
\begin{equation}\label{formaQ}
\begin{split}
Q[\varphi]:=\int_{\R^N}\left( |\nabla \varphi|^2 -cV_n(x)\varphi^2 \right)d\mu, 
\qquad \varphi\in H^1_\mu,
\end{split}
\end{equation}
where $V_n(x)=\sum_{i=1}^{n}\frac{1}{|x-a_i|^2}$.

Consider a partition of unity  $\left\lbrace J_i \right\rbrace_{i=1}^{n+1}$ 
satisfying (\ref{omega}) such that $J_i(x)=J(|x-a_i|/r_0)$ for all 
$x\in \R^N$, $i=1,\dots , n$, with $J$ as in (\ref{PartitionJ}), 
${\rm supp}(J_i)=\overline{B(a_i,r_0)}$.
Then $|x-a_i|\ge r_0$ 
in $\overline{B(a_j,r_0)}$ for
 $i\neq j$, and $V_n(x)\le\frac{n}{r_0^2}$ on $\R^N \setminus \bigcup_{i=1}^n \overline{B(a_i,r_0)}$. 

By virtue of Lemma \ref{lemma2BDE} we are able to write (\ref{formaQ}) as follows
\begin{equation}\label{formaQ2}
Q[\varphi]=\sum_{i=1}^{n} Q[J_i\varphi] + R_n, \qquad \varphi\in H^1_\mu,
\end{equation}
where
$$R_n =\int_{\R^N}|\nabla (J_{n+1}\varphi)|^2\, d\mu 
-c\int_{\R^N} V_n|J_{n+1}\varphi|^2 \,d\mu-\sum_{i=1}^{n+1}
\int_{\R^N}|\nabla J_i|^2\varphi^2\,d\mu.$$
Thanks to the property d) we have
\begin{equation*}
\begin{split}
R_n &= \int_{\R^N}|\nabla (J_{n+1}\varphi)|^2\, d\mu 
-c\int_{\R^N} V_n\left(1-\sum_{i=1}^{n}J_i^2 \right)\varphi^2 \,d\mu
\\& \quad- \sum_{i=1}^{n}\int_{\R^N}\frac{|\nabla J_i|^2}{1-J_i^2}\varphi^2\,d\mu\\
&\ge -c\int_{\R^N}V_n\left( 1-\sum_{i=1}^{n}J_i^2\right) \varphi^2\,d\mu -\sum_{i=1}^{n}
\int_{\R^N}\frac{|\nabla J_i|^2}{1-J_i^2}\varphi^2\,d\mu.
\end{split}
\end{equation*}
Moreover, using the condition (\ref{omega}) we get
\begin{equation*}
R_n \ge-\sum_{i=1}^{n}\int_{B(a_i,r_0)}\left[ \frac{|\nabla J_i|^2}{1-J_i^2}
+c\left(1-J_i^2\right)V_n(x)\right]\varphi^2 \,d\mu-
\frac{c\,n}{r_0^2}\int_{\R^N \setminus \bigcup_{i=1}^n \overline{B(a_i,r_0)}}\varphi^2\,d\mu. 
\end{equation*}
For every $i=1, \dots, n$ 
we can apply Lemma \ref{lemma3BDE}
on $B(a_i,r_0)$ with $(a_i,a_j)=(-a,a)$ up to a change of coordinates for some $j\ne i$.
Considering the partition $\left\lbrace J_i, J_j, \sqrt{1-J_i^2-J_j^2} \right\rbrace $ 
and taking into account that $J_j\equiv 0$ on $B(a_i,r_0)$, we get
\begin{equation}
\begin{split}\label{Rn}
R_n\ge& -\sum_{i=1}^{n} \int_{B(a_i,r_0)} \left[ \frac{k_0+2c}{r_0^2}
+c(1-J_i^2)\left( \sum_{k\ne i,j}\frac{1}{|x-a_k|^2} \right)  \right] 
\varphi^2 \, d\mu 
\\&-\frac{c\,n}{r_0^2}\int_{\R^N \setminus \bigcup_{i=1}^n \overline{B(a_i,r_0)}}\varphi^2 \, d\mu\\
\ge 
&-\sum_{i=1}^{n} \int_{B(a_i,r_0)} \left[ \frac{k_0+2c}{r_0^2}+\frac{(n-2)c}{r_0^2}(1-J_i^2) \right] 
\varphi^2 \, d\mu \\
&-\frac{c\,n}{r_0^2}\int_{\R^N \setminus \bigcup_{i=1}^n \overline{B(a_i,r_0)}}\varphi^2 \, d\mu,
\end{split}
\end{equation}
where $k_0\in [0, \pi^2)$, since we can bound $\frac{1}{|x-a_k|^2}$ 
by $\frac{1}{r_0^2}$ for all $k\neq i,j$.
Taking into account (\ref{formaQ2}) and using the unipolar Hardy inequality (\ref{wHi-c_i}), 
which holds under our assumptions with respect to each pole $a_i\in \R^N$, $i=1,\dots ,n$, we obtain
\begin{equation*}
\begin{split}
Q[J_i \varphi]=&\int_{\R^N}|\nabla J_i\varphi|^2\, d\mu - c\int_{\R^N}\Biggl( \frac{1}{|x-a_i|^2}+
\sum_{\substack{j=1\\j\neq i}}^{n}\frac{1}{|x-a_j|^2} \Biggr)|J_i\varphi|^2\, d\mu \\
\ge& -\left[ k_1+\frac{(n-1)c}{r_0^2}\right] \int_{B(a_i,r_0)} |J_i\varphi|^2 \, d\mu, 
\end{split}
\end{equation*}
from which
\begin{equation}\label{sommaQj}
\sum_{i=1}^{n}Q[J_i \varphi] \ge - k_1 \sum_{i=1}^{n}\int_{B(a_i,r_0)} \varphi^2 \, d\mu -
\frac{(n-1)c}{r_0^2}\sum_{i=1}^{n} \int_{B(a_i,r_0)} J_i^2\varphi^2 \, d\mu 
\end{equation}
From (\ref{formaQ2}), (\ref{Rn}) and (\ref{sommaQj}) we deduce
\begin{equation*}
\begin{split}
Q[\varphi]\ge &- \sum_{i=1}^{n}\int_{B(a_i,r_0)}\left[\frac{k_0+2c}{ r_0^2}
+\frac{(n-2)c}{ r_0^2}(1-J_i^2)+
k_1+\frac{(n-1)c}{ r_0^2}J_i^2\right] \varphi^2 \, d\mu\\
& - \frac{c\,n}{r_0^2}\int_{\R^N \setminus \bigcup_{i=1}^n \overline{B(a_i,r_0)}}\varphi^2 \, d\mu.
\end{split}
\end{equation*}
Since
\begin{equation*}
k_0+2c+c(n-2)(1-J_i^2)+c(n-1)J_i^2 = k_0+cn+cJ_i^2\le k_0+c(n+1),
\end{equation*}
we finally obtain
\begin{equation*}
\begin{split}
Q[\varphi]\ge & -\left[\frac{k_0+(n+1)c}{ r_0^2}+k_1\right] \int_{\Omega}\varphi^2 \, d\mu - 
\frac{c\,n}{r_0^2}\int_{\R^N \setminus \bigcup_{i=1}^n \overline{B(a_i,r_0)}}\varphi^2 \, d\mu\\
\ge& -\left[\frac{k_0+(n+1)c}{ r_0^2}+k_1\right] \int_{\R^N}\varphi^2 \, d\mu ,
\end{split}
\end{equation*}
from which we get inequality (\ref{MwhiGeneral}).
\end{proof}

\bigskip

\section{A class of weight functions}

A class of weight functions satisfying hypotheses $H_1)$ and $H_2)$
is the following
\begin{equation}\label{example mu}
\mu(x)=\frac{e^{-\delta\sum_{j=1}^{n}|x-a_j|^m}}{|x-a_1|^\gamma\cdot
 \cdots \cdot |x-a_n|^\gamma},\qquad \delta\ge 0, \quad \gamma<N-2,\quad m\le 2.
\end{equation}
For $\gamma=0$, $\delta\ne0$ and $m=2$ we get the Gaussian function. 

Taking into account that out of the ball $B(a_i , r_0)$ the term $\frac{1}{|x-a_j|}$ is bounded
and the balls are disjoined, we can see that
the function $\mu$ satisfies $H_1)$ if $\gamma>-N$.
In order to verify $H_2)$, with $\beta=-\alpha$, $\alpha<0$, 
we proceed in the following way.

We observe that, if $\mu_j=\frac{e^{-\delta|x-a_j|^m}}{|x-a_j|^\gamma}$, then
$$
\frac{\nabla\mu}{\mu}=\sum_{j=1}^{n}\frac{\nabla\mu_j}{\mu_j}=
\sum_{j=1}^{n}\left( -\gamma-\delta m |x-a_j|^m\right) \frac{(x-a_j)}{|x-a_j|^2}.
$$
Starting from $H_2)$ and using (\ref{by double product}) we get
\begin{equation}\label {cond for example}
\begin{split}
-\alpha\sum_{i=1}^{n}&\frac{(x-a_i)}{|x-a_i|^2}\cdot\frac{\nabla \mu}{\mu}
= 
\sum_{i=1}^{n}\sum_{j=1}^{n}
\left( -\alpha\gamma-\alpha\delta m |x-a_j|^m\right) 
\frac{(x-a_i)\cdot(x-a_j)}{|x-a_i|^2|x-a_j|^2} 
\\&= 
\sum_{i=1}^{n}\sum_{j=1}^{n}
\left( -\alpha\frac{\gamma}{2}-\alpha\frac{\delta m}{2}|x-a_j|^m\right)
\left(1+\frac{|x-a_i|^2-|a_i-a_j|^2}{|x-a_j|^2}\right)
\\&\le 
k_1+ \sum_{i=1}^{n}\frac{k_2 \alpha}{|x-a_i|^2|x-a_i|^2}.
\end{split}
\end{equation}

\noindent In $B(a_k, r_0)$, for any $k$,  we isolate the term with $i=k$, so the condition $H_2)$ takes the form
\begin{equation}\label{ineq Bk}
\begin{split}
&\frac{-\alpha \gamma 
-\alpha\delta m|x-a_k|^m}{|x-a_k|^2}+
\sum_{i\ne k}\frac{-\alpha \gamma -\alpha\delta m|x-a_i|^m}{|x-a_i|^2}
\\&+
\frac{1}{|x-a_k|^2}\sum_{j\ne k}
\frac{
\left(-\alpha \frac{\gamma}{2} -\frac{\alpha\delta m}{2}|x-a_j|^m\right)
\left(|x-a_j|^2+|x-a_k|^2-|a_k-a_j|^2\right)}
{|x-a_j|^2}
\\&+
\sum_{i\ne k}\frac{1}{|x-a_i|^2}\sum_{j\ne i}
\frac{
\left(-\alpha \frac{\gamma}{2} -\frac{\alpha\delta m}{2}|x-a_j|^m\right)
\left({|x-a_j|^2+|x-a_i|^2-|a_i-a_j|^2}\right)}
{|x-a_j|^2}
\\&=
J_1+J_2+J_3+J_4\le
k_1+\frac{k_2\alpha}{|x-a_k|^2}+\sum_{i\ne k}\frac{k_2\alpha}{|x-a_i|^2}.
\end{split}
\end{equation}

We observe that, in $B(a_k, r_0)$,
$$
|x-a_k|\le r_0, \qquad
r_0\le |x-a_j|\le r_0+|a_k-a_j|\qquad\forall\, j\ne k
$$
then
$$
J_2+J_4\le \sum_{i\ne k}\frac{k_1}{n-1} + \sum_{i\ne k}\frac{k_2\alpha}{|x-a_i|^2}
$$
for $k_1$ large enough.
On the other hand
\begin{equation}\label{Bk}
\begin{split}
(J_1&+J_3)|x-a_k|^2\le 
-\alpha \gamma \left( 1+\frac{n-1}{2}- 
\frac{1}{2}\sum_{j\ne k}\frac{|a_k-a_j|^2}{|x-a_j|^2}\right) 
-\alpha k_2
\\&
-\alpha\frac{\delta m}{2} \sum_{j\ne k}|x-a_j|^m\left(1-\frac{|a_k-a_j|^2}{|x-a_j|^{2}}\right)
-\left(\frac{k_1}{n}+\alpha \frac{\gamma}{2} 
\sum_{j\ne k}\frac{1}{|x-a_j|^2}\right.
\\&
\left.+\alpha\frac{\delta m}{2} \sum_{j\ne k}\frac{1}{|x-a_j|^{2-m}}\right)|x-a_k|^2
-\alpha \delta m |x-a_k|^m  \le 0.
\end{split}
\end{equation}
We observe that when $x$ is near to the pole $a_k$ the contribution of the 
other poles tends to zero.

To estimate the term with of $|x-a_j|^m$  we use the relation 
$$
|x-a_j|\le |x-a_i|+|a_i-a_j| \qquad\forall\, i,j\in\{1,\dots n\}.
$$
Then we get
\begin{equation}\label{alpha delta m}
\begin{split}
-\alpha\frac{\delta m}{2}& \sum_{j\ne k}|x-a_j|^m  \left(1-
\frac{|a_k-a_j|^2}{|x-a_j|^{2}}\right)
\\& \le
-\alpha\frac{\delta m}{2} \sum_{j\ne k}\left(|x-a_k|+|a_k-a_j|\right)^m
\left[1-\frac{|a_k-a_j|^2}{(|x-a_k|+|a_k-a_j|)^2}\right].
\end{split}
\end{equation}
If $ |x-a_k|\le \rho$, $\rho\le r_0$, the last term in 
(\ref{alpha delta m}) 
can be estimated by
$$
-\alpha\frac{\delta m}{2} \sum_{j\ne k}\left(\rho+|a_k-a_j|\right)^m
\left[1-\frac{|a_k-a_j|^2}{(\rho+|a_k-a_j|)^2}\right]:=-\alpha\frac{\delta m}{2}c_\rho
$$
observing that $c_\rho$ tends to zero when $\rho$ goes to zero.
Then inequality  (\ref{Bk})  is satisfied for $k_1$ large enough, with $\rho$ small enough, and
$$
0\le\gamma\le -\frac{k_2+\frac{\delta m}{2}c_\rho}{1+\frac{n-1}{2}+c_1}
\quad \hbox{and}\quad
 -\frac{k_2+\frac{\delta m}{2}c_\rho}{1+\frac{n-1}{2}+c_1}\le\gamma<0.
$$
where
$$
c_1= \left\{\begin{array}{ll}
-\frac{1}{2}\sum_{j\ne k}\frac{|a_k-a_j|^2}{(r_0+|a_k-a_j|)^2} \>&\text{ if } \> \gamma>0 \\
-\frac{1}{2}\sum_{j\ne k}\frac{|a_k-a_j|^2}{r_0^2}
 \> &\text{ if } \> \gamma<0,
\end{array}
\right. 
$$


\noindent Far enough away from the other poles $a_j$, 
with $j\ne k$, and for $ |x-a_k|\ge \rho$,
 the condition $H_2)$ is connected to the inequality 
\begin{equation}\label{out of union Bk}
-\alpha \gamma\left( 1+\frac{n-1}{2}+c_2\right) 
-\alpha k_2-\left(\frac{k_1}{n}+c_3\right)|x-a_k|^2
-\alpha\delta m c_4|x-a_k|^m\le 0
\end{equation}
where the constant $c_2, c_3$ and $c_4$ are so defined:

$$
c_2= \left\{\begin{array}{ll}
0 \>&\text{ if } \> \gamma>0 \\
c_1
 \> &\text{ if } \> \gamma<0
\end{array}
\right. 
,\quad
c_3= \left\{\begin{array}{ll}
\alpha \frac{\gamma}{2}\frac{n-1}{r_0^2}+
\alpha \frac{\delta m}{2}\frac{n-1}{r_0^{2-m}}
&\text{ if } \> \gamma>0 \\
\alpha \frac{\delta m}{2}\frac{n-1}{r_0^{2-m}}
\> &\text{ if } \> \gamma<0
\end{array}
\right. 
,\quad
c_4=1+\frac{c_\rho}{2\rho^m}.
$$
The inequalities (\ref{Bk}) and (\ref{out of union Bk}) 
are both verified if $k_1$ is large enough, $\rho$ small enough, and
$$
0\le\gamma\le -\frac{k_2}{1+\frac{n-1}{2}}\quad \hbox{and}\quad
 -\frac{k_2+\frac{\delta m}{2}c_\rho}{1+\frac{n-1}{2}+c_1}\le\gamma<0.
$$
In order to verify $H'_2)$ we start with the analogous of  (\ref{cond for example}) 
\begin{equation}\label{con with i}
\frac{\nabla f_{\varepsilon, i}}{f_{\varepsilon, i}}\cdot\frac{\nabla \mu}{\mu}=
\frac{\nabla f_{\varepsilon, i}}{f_{\varepsilon, i}}\cdot
\sum_{j=1}^{n}\frac{\nabla\mu_j}{\mu_j}
\le k_1+ \frac{k_2 \alpha}{\varepsilon +|x-a_i|^2}.
\end{equation}
and
reason as in the previous case in
$B(a_i , r_0)$, for any $i\in\{1,\dots, n\}$.

\bigskip

\section{Optimality of the constant}

In order to get the optimality of the constant on the left-hand side in the multipolar Hardy 
inequality we need a further assumption on the function $\mu$.

So we assume that
\begin{itemize}
\item [$H_3)$] 
there exists $i\in \{1,\dots,n\}$ such that 
$$
\sup\left\lbrace \delta \in \R : \frac{1}{|x-a_i|^{\delta}}\in L_{loc}^1(\R^N, d\mu)\right\rbrace =N+k_2.
$$
\end{itemize}
The above condition allows us to estimate the bottom of the spectrum of $-(L+V)$
in a suitable way.

Now we can state the optimality result.

\bigskip

\begin{thm}\label{thm-opt}
	In the hypotheses of Theorem \ref{Hardy via IMS} and if $H_3)$ holds,
	for \break $c>c_o(N+k_2)=\left( \frac{N+k_2-2}{2} \right)^2$ 
	the inequality (\ref{MwhiGeneral}) doesn't hold for any $\varphi \in H_\mu^1$.
\end{thm}

\begin{proof}
	Let us fix a pole $a_i$ such that $H_3)$ holds. Let $\theta \in C_c^\infty(\R^N)$ a cut-off 
	function, $0 \le \theta \le 1$, 
	$\theta=1$ in $B(a_i, 1)$ and $\theta=0$ in $B(a_i, 2)^c$.
	We introduce the function
	$$
	\varphie(x)= \left\{\begin{array}{ll}
	(\varepsilon+|x-a_i|)^\eta \quad &\text{ if } |x-a_i|\in [0,1[,\\
	(\varepsilon+|x-a_i|)^\eta \theta(x) \quad &\text{ if } |x-a_i|\in [1,2[,\\
	0 \quad &\text{ if } |x-a_i|\in [2,+\infty[,
	\end{array}
	\right. $$
	where $\varepsilon >0$ and the exponent $\eta$ is such that
	$$
	\max \left\lbrace -\sqrt{c},-\frac{N+k_2}{2} \right\rbrace< \eta
	< \min \left\lbrace -\frac{N+k_2-2}{2}, 0 \right\rbrace. 
	$$
	The function $\varphie $ belongs to $H_\mu^1$ for any $\varepsilon >0$.
	
	For this choice of $\eta$ we obtain $\eta^2 < c$,  $|x|^{2\eta}\in L_{loc}^1(\R^N,d\mu)$
	and $|x|^{2\eta-2}\notin L_{loc}^1(\R^N,d\mu)$.
	
	Let us assume that $c>c_o(N+k_2)$.   
	Our aim is to prove that the bottom of the spectrum of the operator $-(L+V)$ 
	\begin{equation}\label{lambda1}\displaystyle
	\lambda_1=\inf_{\varphi \in H^1_\mu\setminus \{0\}}
	\left(\frac{\int_{{\mathbb R}^N}|\nabla \varphi |^2\,d\mu
		-c\sum_{j=1}^{n}\int_{{\mathbb R}^N}
		\frac{\varphi^2}{|x-a_j|^2}\,d\mu}{\int_{{\mathbb R}^N}\varphi^2\,d\mu}
	\right).
	\end{equation}
	is $-\infty$. For this purpose we estimate at first the numerator in (\ref{lambda1}).
	
	\begin{equation}\label{numerator}
	\begin{split}
	\int_{\R^N}& \left( |\nabla \varphie|^2 -\sum_{j=1}^{n}
	\frac{c}{|x-a_j|^2}\varphie^2 \right) \, d\mu =
	\\ &=
	\int_{B(a_i,1)}\left[ |\nabla(\varepsilon +|x-a_i|)^\eta|^2 
	-\sum_{j=1}^{n}\frac{c}{|x-a_j|^2}(\varepsilon +|x-a_i|)^{2\eta}\right]\, d\mu 
	\\ &
	+\int_{B^c(a_i,1)}\left[ |\nabla(\varepsilon +|x-a_i|)^\eta \theta|^2
	-\sum_{j=1}^{n}\frac{c}{|x-a_j|^2}(\varepsilon +|x-a_i|)^{2\eta}\theta^2 \right]\, d\mu
	\\ & \le
	\int_{B_{(a_i,1)}}\left[\eta^2(\varepsilon +|x-a_i|)^{2\eta -2}
	-\frac{c}{|x-a_i|^2}(\varepsilon +|x-a_i|)^{2\eta} \right]\, d\mu\\
	&+\eta^2\int_{B^c(a_i,1)}  (\varepsilon +|x-a_i|)^{2\eta -2}\theta^2\, d\mu 
	+ \int_{B^c(a_i,1)}(\varepsilon +|x-a_i|)^{2\eta}|\nabla \theta|^2\, d\mu
	\\&+
	2\eta\int_{B^c(a_i,1)} \theta (\varepsilon +|x-a_i|)^{2\eta-1}\frac{x-a_i}{|x-a_i|}\cdot 
	\nabla \theta\, d\mu
	\\ & \le
	\int_{B(a_i,1)} (\varepsilon +|x-a_i|)^{2\eta} \left[ \frac{\eta^2}{(\varepsilon +|x-a_i|)^2}
	-\frac{c}{|x-a_i|^2}\right]  \,d\mu
	\\&+ 
	2\eta^2\int_{B^c(a_i,1)}(\varepsilon +|x-a_i|)^{2\eta-2}\theta^2\,d\mu 
	+ 2\int_{B^c(a_i,1)}(\varepsilon +|x-a_i|)^{2\eta}|\nabla \theta|^2\,d\mu
	\\ & \le
	\int_{B(a_i,1)} (\varepsilon +|x-a_i|)^{2\eta} \left[ \frac{\eta^2}{(\varepsilon +|x-a_i|)^2} 
	-\frac{c}{|x-a_i|^2}\right]  \,d\mu +C_1,
	\end{split}
	\end{equation}
	where $C_1=2\left(\eta^2+ \|\nabla \theta\|_\infty \right)\int_{B^c(a_i,1)}d\mu$.
	
	Furthermore 
	\begin{equation}\label{denominator}
	\int_{\R^N}\varphie^2\,d\mu \ge \int_{B(a_i,2)\setminus B(a_i,1) }(\varepsilon +
	|x-a_i|)^{2\eta}\theta^2\,d\mu =C_{2,\varepsilon}.
	\end{equation}
	Putting together (\ref{numerator}) and (\ref{denominator}) we get from (\ref{lambda1})
	$$
	\lambda_1 \le \frac{\int_{B(a_i,1)}(\varepsilon +|x-a_i|)^{2\eta}\left[ \frac{\eta^2}
		{(\varepsilon +|x-a_i|)^2}-\frac{c}{|x-a_i|^2} \right] \,d\mu+C_1}{C_{2,\varepsilon}}.
	$$
	Letting $\varepsilon\to 0$ in the numerator above, taking in mind that 
	$|x-a_i|^{2\eta}\in L^1_{loc}(\R^N, d\mu)$
	and Fatou's lemma, we obtain
	\begin{equation*}
	\begin{split}
	\lim\limits_{\varepsilon\to 0}\int_{B(a_i,1)}(\varepsilon +|x-a_i|)^{2\eta}\left[ \frac{\eta^2}
	{(\varepsilon +|x-a_i|)^2}-\frac{c}{|x-a_i|^2} \right] \,d\mu & \\
	\le -(c-\eta^2)\int_{B(a_i,1)}|x-a_i|^{2\eta-2}\,d\mu &=-\infty
	\end{split}
	\end{equation*}
	and, then,  $\lambda_1=-\infty $.
\end{proof}

\bigskip

\section{Existence and nonexistence results}

We say that $u$ is a weak solution to the problem $ (P)$ 
if, for each $T, R>0 $, we have
$$u\in C(\left[ 0, T \right] , L^2_\mu ), \quad Vu\in L^1(B(0,R) \times \left( 0,T\right) , d\mu dt )$$
and
\begin{equation}\label{weak-sol}
\int_0^T\int_{\R^N}u(-\partial_t\phi - L\phi )\,d\mu dt -\int_{\R^N}u_0\phi(\cdot ,0)\,d\mu =
\int_0^T\int_{\R^N}Vu\phi \, d\mu dt
\end{equation}
for all $\phi \in W_2^{2,1}(\R^N \times \left[ 0,T\right])$ having compact support with 
$\phi(\cdot ,T)=0$.\\

For any domain $\Omega\subseteq \R^N$, $ W_2^{2,1}(\Omega\times (0,T)) $ is the parabolic Sobolev space 
of the functions $u\in L^2(\Omega \times (0,T)) $ having weak space derivatives $D_x^{\alpha}
u\in L^2(\Omega \times (0,T))$ for $|\alpha |\le 2$ and weak time derivative
$\partial_t u \in L^2(\Omega \times (0,T))$ equipped with the norm 
\begin{equation*}
\begin{split}
\|u\|_{W_2^{2,1}(\Omega\times (0,T))}:= \Biggl( 
\|u\|_{L^2(\Omega \times (0,T))}^2 +& \|\partial_t u\|_{L^2(\Omega \times (0,T))}^2 \\
&\, + \sum_{1\le |\alpha |\le 2} \|D^{\alpha}u\|_{L^2(\Omega \times (0,T))}^2
\Biggr)^{\frac{1}{2}}.
\end{split}
\end{equation*}
In order to investigate on existence and nonexistence of positive weak solution to the 
evolution problem $(P)$ using multipolar weighted Hardy inequalities, we need to state some 
preliminary results regarding the operator $L$, its associated semigroup, and the space $H_{\mu}^1$.
These results will allow us to state existence and nonexistence conditions using the Cabr\'e-Martel's 
approach.

\bigskip

Let us assume that the function $\mu$ is a weight function on $\R^N$, $\mu >0$. 
In the hypothesis
$\mu \in C_{loc}^{1,\lambda}(\R^N)$, $\lambda\in(0,1)$
it is known that the operator $L$ with domain
$$D_{max}(L)=\lbrace u\in C_b(\R^N)\cap W_{loc}^{2,p}(\R^N) \text{ for all }  1<p<\infty, 
Lu\in C_b(\R^N)\rbrace$$
is the weak generator of a not necessarily $C_0$-semigroup in $C_b(\R^N)$. 
Since $\int_{\R^N}Lu\,d\mu=0$ for any $u\in C_c^{\infty}(\R^N)$,
then $d\mu=\mu (x)dx$ is the invariant measure for this semigroup in $C_b(\R^N)$. 
So we can extend it to a positivity preserving and analytic $C_0$-semigroup $\lbrace T(t)\rbrace_{t\ge 0}$ 
on $L^2_\mu$, whose generator is still denoted by $L$ (see \cite{BertoldiLorenzi}).

\medskip

In the more general setting, when the assumptions on $\mu$ allow degeneracy at some points, 
we require the further conditions
to get $L$ generates a semigroup. In particular we assume
\begin{itemize}
	\item [$H_4)$] $\mu \in C_{loc}^{1,\lambda}(\R^N\setminus \{a_1, \dots , a_n\})$, $\lambda\in(0,1)$,
	$\mu\in H^{1}_{loc}(\R^N)$,  $\frac{\nabla \mu}{\mu}\in L^r_{loc}(\R^N)$ 
	for some $r>N$, and $\displaystyle \inf_{x\in K}\mu (x)>0$ for any compact set $K\subset \R^N$.
\end{itemize}
So by \cite[Corollary 3.7]{alb-lor-man}), we have that the closure of 
$(L,C_c^{\infty}(\R^N))$ on $L^2_\mu$ generates
a strongly continuous and analytic Markov semigroup $\lbrace T(t)\rbrace_{t\ge 0}$  on $L^2_\mu$.

For such a semigroup $\{T(t)\}_{t\ge 0}$ and its generator $L$ there are 
some interesting properties which we list in the Proposition below.
We omit the proof since it is analogous to \cite[Proposition 2.1]{CGRT}.

\medskip

\begin{prop}\label{propH1} 
	Assume that $\mu$ satisfies $H_4)$. Then the following assertions hold:
	\begin{itemize}
		\item [$1)$] $D(L)\subset H^1_{\mu}$.
		\item [$2)$] For every $f\in D(L),\, g\in H^1_{\mu}$ we have
		\[\int L f g \,d\mu=-\int \nabla f\cdot \nabla g\,d\mu.\]
		\item [$3)$] $T(t)L^2_\mu\subset D(L)$ for all $t>0$.
	\end{itemize}
\end{prop}

\bigskip
Now we prove two general results, which state 
the density of $C_c^{\infty}\left( \R^N\setminus\{a_1, \dots, a_n\} \right)$ in $W_\mu^{1,p}$, 
$1\le p<\infty$. Note that, if $p=2$, under assumptions $ii)$ and $iii)$ in $H_1)$, the space 
$W_\mu^{1,2}$ coincides with $H_\mu^1$ (see \cite[Corollary 1.2]{T}).

\medskip

Let us set $L^p_\mu:=L^p(\R^N,d\mu)$ and 
$\|u\|_{p,\mu}:=\left( \int_{\R^N}|u|^pd\mu \right)^{\frac{1}{p}}$, $1\le p<\infty$.

We state the following Proposition.
\begin{prop}
Let $W_\mu^{1,p}=\overline{C_c^{\infty}(\R^N)}^{\|\cdot\|_{W_\mu^{1,p}}}$
where $\|u\|_{W_\mu^{1,p}}=\|u\|_{p,\mu}+\|\nabla u\|_{p,\mu}$.
If
\begin{equation}\label{eq:cond1}
\lim_{\delta\to 0}\frac{1}{\delta^p}\int_{B(a_i,\delta)}d\mu=0 
\qquad \text{for any}\quad i=1, \dots , n
\end{equation} 
then $C_c^{\infty}(\R^N\setminus\{a_1, \dots, a_n\})$ is dense in $W^{1,p}_{\mu}$. 
\end{prop}
\begin{proof}
Our aim is to approximate $u\in C_c(\R^N)$ with functions in 
$C_c^{\infty}(\R^N\setminus\{a_1, \dots, a_n\})$ with respect to the norm
$\|\cdot \|_{W_\mu^{1,p}}$.

Let 
$$
\vartheta= \left\{\begin{array}{ll}
0 \quad &\text{ in } \bigcup_{i=1}^n B(a_i, \frac{r_0}{2}),\\
\phi_1 \quad &\text{ in }  \overline{B(a_1,r_0)}\setminus B(a_1, \frac{r_0}{2}),\\
\vdots & \quad\quad\quad\quad \vdots \\
\phi_n \quad &\text{ in }  \overline{B(a_n,r_0)}\setminus B(a_n, \frac{r_0}{2}),\\
1 \quad &\text{ in } \R^N \setminus\bigcup_{i=1}^n B(a_i,r_0),
\end{array}
\right. 
$$
where $\phi_i\in C_b^\infty(R^N)$ for any $i\in \{1,\dots , n \}$,
such that $\phi_i=0$ on $\partial B(a_i, \frac{r_0}{2})$ and
$\phi_i=1$ on $\partial B(a_i, r_0)$.

We observe that $\vartheta_k(x)=\vartheta(kx)$ belongs to
$ C_c^{\infty}(\R^N\setminus\{a_1, \dots, a_n\})$, $\vartheta_k\to 1$ 
pointwisely in $\R^N\setminus\{a_1, \dots, a_n\}$ and
$\|\nabla \vartheta_k\|_{\infty}\leq Ck$.
So we get
\begin{align*}
\|u-(u\vartheta_k)\|^{p}_{W_\mu^{1,p}}\leq
C\left(\|u(1-\vartheta_k)\|^{p}_{p,\mu}+\|\nabla \left( u(1-\vartheta_k)\right)\|^{p}_{p,\mu}\right).
\end{align*}
The first term on the right-hand side converges to $0$ by dominated convergence. As regards the 
second one we have 
\begin{align*}
\|\nabla u(1-\vartheta_k)\|^{p}_{p,\mu}
  &\leq 
  C\left(\int_{\R^N}(1-\vartheta_k)^p|\nabla u|^pd\mu 
  +\int_{\R^N}|\nabla \vartheta_k|^p | u|^pd\mu\right)\\
&\leq C\left( \int_{\R^N}(1-\vartheta_k)^p|\nabla u|^pd\mu 
+k^p\int_{\bigcup_{i=1}^n B(a_i,r_0/k)}|   u|^pd\mu\right)\\
&\leq C\left(\int_{\R^N}(1-\vartheta_k)^p|\nabla u|^pd\mu 
+k^p\|u\|_{\infty}^p\int_{\bigcup_{i=1}^n B(a_i,r_0/k)} d\mu\right).\\
\end{align*}
To get the result we observe that the first integral converges to $0$ by dominated convergence, 
the last one by condition \eqref{eq:cond1}.

\end{proof}

Now we prove the density result.
\begin{prop}
Let $p<N$. If $\mu \in W^{1,p}_{loc}\left( \R^N \right)$ then
$C_c^{\infty}\left( \R^N\setminus\{a_1, \dots, a_n\} \right)$ is dense in $W_\mu^{1,p}$.
\end{prop}
\begin{proof}
We have $\mu\in L_{loc}^{p^*}(\R^N)$ where $p^*=\frac{Np}{N-p}$ is the Sobolev exponent of $p$.
It suffices to verify  condition \eqref{eq:cond1}. Then,
for any $i=1, \dots , n$ 
\begin{align*}
\frac{1}{\delta^p}\int_{B(a_i,\delta)}\mu dx
\leq \frac{1}{\delta^p}\left( \int_{B(a_i,\delta)}\mu^{p^*}\,dx \right)^{\frac{1}{p^*}}
\left( \int_{B(a_i,\delta)}dx \right)^{\frac{1}{(p^*)'}} 
\leq C\delta ^{\frac{N}{(p^*)'}-p},
\end{align*}
where $\frac{1}{p^*}+\frac{1}{(p^*)'}=1$.
One can easily verify that $\frac{N}{(p^*)'}-p>0$ if $p<N.$
\end{proof}

\bigskip

Using the density of $C_c^{\infty}(\R^N\setminus\{a_1, \dots, a_n\})$ in $H^1_\mu$ we are able to 
prove the following Lemma for compact sets contained in $\R^N\setminus\{a_1, \dots, a_n\}$. 
The result allows us to extend the Cabr\'e-Martel's approach to the case of weight function 
having many singularities stating an estimate for a weak solution to the problem $(P)$ 
(cf. \cite[Theorem 2.1]{GGR}). 
The proof makes use of the same technique as in \cite[Lemma 2.2] {CGRT} 
in the case of one singularity.
\begin{lemma}\label{lm:strictly-positivity}
Let $V$ be a positive function belonging to $L^1_{loc}(\R^N)$. Let $u$ be a weak solution of $(P)$. 
Then, for every compact set $K\subset \R^N\setminus\{a_1, \dots, a_n\}$ and $t>0$
there exists $c(t)>0$ (not depending on $V)$
such that 
$$u(t,x)\geq c(t)\int _{K}u_0 \,d\mu \qquad\text{on}\quad K\times [0,T].$$
\end{lemma}
\begin{proof} Let $u_0 \in C_c^{\infty}(\R^N)$ and let $u$ be a weak solution of $(P)$.
Let
$C_R=B(0,R)\setminus \bigcup_{i=1}^n \overline{B(a_i,1/R)}$, with $R$ large enough, such that 
$K\subset C_R$ and let $\varphi\in C_c^{\infty}(C_R)$ such that $0\leq \varphi\leq 1$.

Consider the problem
\[(Pb)\,
 \left\{
 \begin{array}{ll}
 v_t(x,t)=Lv(x,t),& \text{ on }C_R\times (0,T],\\
 v(x,t)=0,&\text{ on } \partial C_R,\\
 v(x,0)=\varphi u_0.
 \end{array}
 \right.
\]
By a classical result,  since $v(x,0)\in C_c^{2+\alpha}(C_R)$, then the problem $(Pb)$ admits a solution 
$v\in C^{2+\alpha,1+\frac{\alpha}{2}}(\overline C_R\times [0,T])$. Moreover,
\[
v(x,t)=\int_{C_R}G(t,x,y)v(y,0)dy 
\]
where $G$ is a strictly positive function on $(0,+\infty)\times C_R\times C_R$.

\noindent Let $\displaystyle c(t)=\min_{(x,y)\in K\times K} G(t,x,y)$. We have for every $x\in K$ 
\[
 v(x,t)\geq \int_{K}G(t,x,y)v(y,0)dy\geq c(t)\int_Kv(y,0)dy.
\]
Furthermore,  $v$ is a weak solution to $v_t=Lv$ in $C_R$. In particular, 
for all $\phi \in W^{2,1}_2 (C_R\times [0,T])$ with $\phi(\cdot,0)\geq0$
having compact support with $\phi(\cdot, T)=0$, we have
$$\int_0^T\int_{C_R}v(-\partial_t\phi -L\phi)\,d\mu\,dt
-\int_{C_R }(\varphi u_0)\phi(\cdot ,0)\,d\mu=0.$$
Comparing with (\ref{weak-sol}), one obtains
\begin{equation}\label{eqA}
\int_0^T\int_{C_R}(v-u)(-\partial_t\phi -L\phi)\,d\mu\,dt
=\int_{C_R }(\varphi u_0-u_0-Vu )\phi(\cdot ,0)\,d\mu\leq 0. 
\end{equation}

Fix $T,\,R>0$, $0\le \psi
\in C_c^\infty(C_R\times [0,T])$ such that ${\rm supp}\,\psi \subset  C_R\times [0,T]$ 
and consider the parabolic problem
$$\left\{\begin{array}{ll}
\partial_t\phi +L\phi =-\psi ,& \hbox{on }C_R\times (0,T),\\
\phi|_{\partial C_R\times (0,T)}= 0,\\
\phi(x,T)=0,& x\in {\mathbb R}^N.
\end{array} \right.
$$
By \cite[Theorem IV.9.1]{lad-sol-ura} we obtain a solution $0\le \phi \in W^{2,1}_2(C_R\times (0,T))$. 
We can insert the solution $\phi$ in (\ref{eqA}). Therefore,
$$\int_0^T\int_{C_R }(v-u)\psi \,d\mu \,dt \le 0$$ for all $0\le \psi \in C_c^\infty(C_R\times [0,T])$. 
Thus,
$$u\geq v\geq c(t)\int_{C_R}\varphi u_0 d\mu.$$
Since the last inequality holds true for every $\varphi \in C_c^\infty (C_R)$ one obtains
$$u\geq c(t)\int_{C_R}u_0 d\mu.$$
\end{proof}

The above results allow us to state the following Theorem by proceeding as in \cite[Theorem 2.1]{GGR}.

\begin{thm}\label{th:cabr-mart-H1} Assume that $\mu$ satisfies the hypothesis $H_4)$ and 
$0\le V\in L_{loc}^1({\mathbb R}^N)$. Then the following hold:
\begin{enumerate}
\item[$1)$] If $\lambda_1(L+V)>-\infty$, then there exists a
positive weak solution \\$u\in C([0,\infty),L^2_\mu)$ of $(P)$
satisfying
\begin{equation}\label{eq222}
\|u(t)\|_{L^2_\mu}\le Me^{\omega t}\|u_0\|_{L^2_\mu},\quad t\ge
0,
\end{equation}
for some constants $M\ge 1$ and $\omega \in {\mathbb R}$.
 \item[$2)$] If
$\lambda_1(L+V)=-\infty$, then for any $0\le u_0\in
L^2_\mu\setminus \{0\},$ there exists no positive weak solution of $(P)$
satisfying \eqref{eq222}.
\end{enumerate}
\end{thm}

From Theorem \ref{Hardy via IMS}, Theorem \ref{thm-opt} and Theorem \ref{th:cabr-mart-H1} 
we get the following existence and nonexistence result.
\begin{thm}
	Assume that the weight function $\mu$ satisfies hypotheses 
	$H_1)$--$H_4)$ and $0\le V(x)\le \sum_{i=1}^n \frac{c}{|x-a_i|^2}$,
$c>0$, $a_i\in\R^N$, $i\in\{1, \dots, n\}$.
	The following assertions hold:
	\begin{enumerate}
		\item[$1)$] If $0< c\le c_o(N+k_2)=\left( \frac{N+k_2-2}{2} \right)^2$, 
		then there exists a positive weak
		solution $u\in C([0,\infty),L^2_\mu)$ of $(P)$
		satisfying
		\begin{equation}\label{eq:est}
		\|u(t)\|_{L^2_\mu}\le Me^{\omega t}\|u_0\|_{L^2_\mu},\quad
		t\ge 0,
		\end{equation}
		for some constants $M\ge 1$, $\omega \in {\mathbb R}$, and any $u_0\in L^2_\mu$. 
		\item[$2)$] If
		$c> c_o(N+k_2)$, then for any $0\le u_0\in L^2_\mu,\,u_0\neq 0,$
		there is no positive weak solution of $(P)$
		with $V(x)=\sum_{i=1}^n \frac{c}{|x-a_i|^2}$ satisfying 
		(\ref{eq:est}).
	\end{enumerate}
\end{thm}

\end{document}